\documentclass{siamltex}

\usepackage{epsfig,graphicx,psfrag,mathrsfs}
\usepackage{amssymb}
\usepackage{amsmath}
\usepackage[algoruled]{algorithm2e}
\usepackage{color}

\usepackage{colortbl}
\usepackage{xcolor}

\usepackage{multirow}

\newtheorem{example}{{\it Example}}

\title{Integer Low Rank Approximation of Integer matrices\\}

\author{
Bo Dong \thanks{ School of Mathematical Sciences, Dalian
University of Technology, Dalian, Liaoning 116024, China
(dongbo@dlut.edu.cn). } \and Matthew M. Lin
\thanks{Corresponding author.
Department of Mathematics, National Chung Cheng University,
Chia-Yi 621, Taiwan. (mhlin@ccu.edu.tw). } \and
 Haesun Park\thanks{
 School of Computational Science and Engineering, Georgia Institute of Technology, Atlanta, GA 30332-0765, USA
 (hpark@cc.gatech.edu).
 }
}

\begin{document}

\maketitle

\begin{abstract}
%Integer data sets frequently appear
%in many applications in sciences and technology. To analyze these data,
%integer matrix factorization
%has received much attention due to its
%capacity of representing the results in integers preserving the meaning of the original data set. Since integers are discrete in nature, conventional techniques for matrix factorization developed for real numbers cannot be directly applied. To the best of our knowledge, integer low rank approximation of
%an integer matrix has not been proposed in the literature
%before. In this paper, we connect this approximation with the so-called $k$-median method and
%conjecture that this is an NP-hard problem. Then as a practical way of tackling this difficult problem, we develop a block coordinate descent method based on a simple and effective way of solving integer
%least squares problems. Numerical experiments on association
%analysis and random integer matrices are presented,
%and the comparison results show our method can find a more
%accurate solution than some of other existing alternatives.

%Integer data sets frequently appear
%in many applications in sciences and technology.

Integer data sets frequently appear in many applications in sciences
and technology. To analyze these, integer low rank approximation has
received much attention due to its capacity of representing the
results in integers preserving the meaning of the original data
sets. To our knowledge, none of previously proposed techniques
developed for real numbers can be successfully applied, since
integers are discrete in nature. In this work, we start with a
{thorough} review of algorithms for solving integer least
squares problems, {and} then develop a block coordinate
descent method based on the integer least squares estimation to
obtain the integer low rank approximation of integer matrices. The
numerical application on association analysis and numerical
experiments on random integer matrices are presented. Our computed
results seem to suggest that our method can find a more accurate
solution than other existing methods for continuous data sets.



\end{abstract}

\section{Introduction}
The study of integer approximation has long been a subject of
interest in many areas such as networks of communication, data organization, environmental chemistry, lattice design, and finance~\cite{Agrell2002, Hassibi1996,Morgan01,Lin2011}. In this work, we consider integer low rank approximation of integer matrices (ILA).
Let $\mathbb{Z}^{m \times n}$ denote the set of $m\times n$ integer matrices. Then, the ILA problem can be stated as follows:
\begin{quote}
{\bf (ILA)} Given an integer matrix $A \in \mathbb{Z}^{m \times
n}$ and a positive integer
%\footnote{The determination
% of such a rank $k$ is tricky and a problem-dependent parameter.}
 %
 $k < \min\{m,n\}$, find $U \in \mathbb{Z}^{m\times k}$ and $V \in
\mathbb{Z}^{k \times n}$
so that the residual function
\begin{equation}~\label{eq:ila}
f(U,V):=\|A-UV\|_F^2
\end{equation}
 is minimized.
\end{quote}
Concretely, {the ILA is to represent an original integer
data set in a space spanned by integer basis vectors using integer
representations minimizing the residual function in~\eqref{eq:ila}.}
One example that characterizes this problem is the
so-called~\emph{market basket transactions}. See
Table~\ref{trandata} for example, where it shows orders from five
customers $C_1,\ldots,C_5$.
\begin{table}[h!!!]
  \caption{An integer representation of the transaction example}  \label{trandata}
\begin{center}
  \begin{tabular}{c|cccccc}
&Bread & Milk & Diapers & Eggs & Chips & Beer\\
\hline
$C_1$&2&1& 3&0&2&5\\
$C_2$&2&1&1&0&2&4\\
$C_3$&0&0&4&2&0&2\\
$C_4$&4&2&2&0&4&8\\
$C_5$&0&0&2&1&0&1\\
  \end{tabular}
  \end{center}
\end{table}
This should be a huge data set with many customers and shopping
items in practice. Similar to the discussion given in~\cite{KG03,
KGR06}, we would like to discover an association rule such as
``$\mbox{\{2 diapers, 1 egg\}}\Rightarrow \mbox{\{1 beer\}}$", which
means that if a customer has shopped for two diapers and one egg,
he/she has a higher possibility of shopping for one beer. The
``possibility" is determined by a weight associated with each
transaction. To keep the number of items to be integer, the weight
is defined by nonnegative integers corresponding to the
representative transaction. For example, we can approximate the
original integer data set, which is define in Table~\ref{trandata},
by two representative transactions ${[2,1,1,0,2,4]}$  and
${[0,0,2,1,0,1]}$ with the weight determined by a $5\times 2$
integer matrix, that
 is,
 \begin{equation}\label{tranLRF}
A = \left[
  \begin{array}{cccccc}
2&1& 3&0&2&5\\
2&1&1&0&2&4\\
0&0&4&2&0&2\\
4&2&2&0&4&8\\
0&0&2&1&0&1\\
  \end{array}\right] \approx
  \left[
  \begin{array}{cc}
1&1\\
1&0\\
0&2\\
2&0\\
0&1\\
  \end{array}
  \right]
\left[
  \begin{array}{cccccc}
2&1&1&0&2&4\\
0&0&2&1&0&1\\
  \end{array}
  \right].
 \end{equation}
This is indeed a rank-two
approximation which provides the possible shopping behaviors of
the original five customers.
%
%\begin{quote}
%{\bf (ILA)} Given an integer matrix $A \in \mathbb{Z}^{m \times
%n}$, find $U \in \mathbb{Z}^{m\times k}$ and $V \in
%\mathbb{Z}^{k \times n}$, subject to a positive integer
% %
% $k < \min\{m,n\}$,
%so that the functional
%\begin{equation*}
%f(U,V):=\|A-UV\|_F^2
%\end{equation*}
% is minimized.
%\end{quote}
%%

Note that the study of the ILA is to analyze original integer
(i.e.,discrete) data in term of integer representatives so that
the underlying information can be conveyed directly. Due to the
discrete characteristic, conventional techniques such as
SVD~\cite{Golub2013} and the nonnegative matrix factorization
(NMF)~\cite{Kawamoto00,Donoho03,
Chu2008,Kim08,Kim2011,Kim2014} are inappropriate and unable to solve this
problem. Also, we can regard this problem as a
generalization of the so-called {Boolean} matrix
approximation, where entries of $A$, $U$ and $V$ are limited
to $\mathbb{Z}_2:=\{0,1\}$, i.e., binaries, and columns of the
matrix $U$ are orthogonal. In~\cite{KG03, KGR06}, Koyut\"{u}rk
et al.~\cite{KG03, KGR06} proposed a nonorthogonal binary
decomposition approach to recursively approximate the given
data by the rank-one approximation. This idea is then
generalized in~\cite{Lin2011} to handle integer data sets, called
the binary-integer low rank approximation (BILA), and can be defined
as follows:

\begin{quote}
{\bf (BILA)} Given an integer matrix $A \in \mathbb{Z}^{m \times
n}$ and a positive integer $k < \min\{m,n\}$,
find $U\in \mathbb{Z}_2^{m\times k}$ satisfying $U^\top U = I_k$, where $I_k$ is a $k\times k$ identity matrix,
and $V \in \mathbb{Z}^{k \times n}$
such that the residual function
\begin{equation*}
f(U,V):=\|A-UV\|_F^2
\end{equation*}
 is minimized.
\end{quote}

To demonstrate the BILA, we use the scenario of selling laptops. We know that each laptop includes $n$
slots which can be equipped with only one type of specified
options, for example, 8GB memory or 1TB hard drive. Here, the uniqueness is determined by the orthogonal columns of $U$. Let $t_i$ represent the number of different types of options for each
slot. This implies that we have $\prod_{i=1}^{n}t_{i}$ possible
cases. To efficiently fulfill different customer orders and
increase the sales, the manager needs to decide a few basic
models, e.g., $k$ basic ones. One plausible approach is to
learn the information from past experience. This amounts to
first labeling every possible option in a slot by an integer
while each integer is corresponding to a particular attribute.
For example, record $m$ past customers' orders  into an $m\times
n$ integer data matrix.  Second, divide the orders
into $k$ different integer clusters. These $k$ clusters then
lead to the $k$ different models that should be provided; see
also~\cite{Morgan01} for a similar discussion in
telecommunications industry.

Besides, the BILA behaves like the $k$-means problem.
That is because all rows of $A$ are now divided into $k$ disjoint clusters
due to the fact that
$U = [u_{ij}]\in\mathbb{Z}_2^{m\times k}$
and  $U^\top U = I_k$, and each centroid is the rounding of the means of the rows from each disjoint cluster~\cite{Lin2011}.
It is known that the $k$-means problem
is NP-hard~\cite{Aloise2009,Dasgupta2009}. This might indicate the
possibility of solving the BILA is NP-hard. Moreover, it is true
that the BILA is a special case of ILA. We thus
conjecture that the ILA is an NP-hard problem.
%
%\begin{theorem}
% The ILA is an NP-hard problem.
%\end{theorem}

To handle this ``seemingly" NP-hard problem, we thus organize
this paper as follows. In Section~\ref{sec:adi}, we investigate
the ILA in terms of the block coordinate descent (BCD) method and discuss some characteristics related to
orthogonal constraints. In Section~\ref{sec:ils}, we review the
methods for solving integer least squares problem, apply them
to solve the ILA, and {present a convergence theorem
correspondingly}. Three examples are demonstrated in
Section~\ref{sec:test} to show the capacity and efficiency of
our ILA approaches and the concluding remarks are given in
Section~\ref{sec:conclusion}.

\section{Block Coordinate
Descent}~\label{sec:adi}
%Due to the NP-hard complexity of solving the ILA, %i.e., the BIF,
%it would be unrealistic to uncover the two optimal matrices $U$
%and $V$ simultaneously. For this reason, our approach for
%solving ILA is based on a series of alternating direction
%iterations (ADI).
The framework of the BCD %solve the ILA, our approach in this work is based on the block coordinate descent (BCD)
method includes the following
procedures:
\begin{enumerate}
\item[Step 1.] Provide an initial value $V\in\mathbb{Z}^{k\times n}$.
\item[Step 2.] Iteratively, solve the following problems until a stopping criterion is satisfied:
\begin{subequations}\label{constraint}
\begin{eqnarray}
\min_{U\in\mathbb{Z}^{m\times k}} \|A - UV\|_F^2,% \mbox { where } V \mbox{ is fixed},
 \label{constraint1}
 \\
\min_{V\in\mathbb{Z}^{k\times n}} \|A - UV\|_F^2.% \mbox { where } U \mbox{ is fixed}.
 \label{constraint2}
\end{eqnarray}
\end{subequations}
\end{enumerate}

Alternatively, we can initialize $U$ first and
iterate~\eqref{constraint1} and ~\eqref{constraint2} in the
reverse order. Note that in Step 2, each subproblem is required
to find the optimal matrices $U$ or $V$ so that $\|A-UV\|_F$ is
minimized. Observe further that
\begin{equation}\label{separate}
\|A - UV\|_F^2 = \sum_{i}\|A(i,:) - U(i,:)V\|_2^2 = \sum_j\|A(:,j) - U V(:,j)\|_2^2.
\end{equation}
It follows that~\eqref{separate} is minimized, if for each $i$,
%\begin{subequations}
\begin{eqnarray*}
\|A(i,:) - U(i,:)V\|_2
\end{eqnarray*}
is minimized and for each $j$,
\begin{eqnarray*}
\|A(:,j) - U V(:,j)\|_2
\end{eqnarray*}
%\end{subequations}
is minimized. Let the symbol ``round(X)" denote a function which rounds each element of $X$ to the nearest integer.
Since the 2-norm is invariant under orthogonal transformations,
%{Applying a similar result as
%orthogonal invariance of the Frobenius norm,}
one immediate
result can be stated as follows.

\begin{theorem}\label{thm:lintu}
Suppose  $\mathbf{a} \in \mathbb{Z}^{1\times n}$ and  $V\in \mathbb{Z}^{k\times n}$ with
$n\geq k$ and $VV^\top = I_k$, where $I_k$ is a $k\times k$ identity matrix. Let
\begin{equation}\label{opt_u}
\hat{\mathbf{u}}= \rm{round} (\mathbf{a}V^\top ) \in\mathbb{Z}^{1\times k}.
\end{equation}
 Then the following is true:
\begin{equation*}
\|\mathbf{a} - \hat{\mathbf{u}} V\|_2  \leq
\min_{\mathbf{u}\in\mathbb{Z}^{1\times k}} \|\mathbf{a} -
\mathbf{u}V\|_2 .
\end{equation*}

\end{theorem}

With an obvious change of notation, an analogous
statement would be true for $\mathbf{a} \in \mathbb{Z}^{m\times 1}$, and $U\in \mathbb{Z}^{m\times k}$ having
$m\geq k$ and $U^\top U = I_k$. Besides, Theorem~\ref{thm:lintu} provides a way to obtain the optimal matrix $U$ in the following way, whose proof is by direct observation.

%each row of $U$.

%
%. %We omit the proof here.
%

%With the formulations obtained in Theorem~\ref{thm:inls} and
%Theorem~\ref{thm:lintu}, we immediately have the following
%results for obtaining the optimal matrices $U$ (respectively,
%$V$), provided that $V$ (respectively, $U$) is given.
%
%
\begin{theorem}\label{thm:intu}
Suppose $A \in \mathbb{Z}^{m\times n}$ and  $V\in \mathbb{Z}^{k\times n}$ with
$n\geq k$ and $VV^\top = I_k$. Let
\begin{equation*}
\widehat{U} = \rm{round} (AV^\top ) \in\mathbb{Z}^{m\times k}.
\end{equation*}
 Then the following is true:
\begin{equation*}
\|A - \widehat{U}V\|_F \leq \min_{U\in\mathbb{Z}^{m\times k}}\|A - UV\|_F.
\end{equation*}

\end{theorem}

Similarly, the optimal matrix $V$ can be obtained immediately, provided that $U$ is given and $U^\top U = I_k$.
But, could the optimal matrices $U$ and $V$ be obtained
simply by rounding real optimal solutions $( {A}V^\top)(V V^\top)^{-1}$ and $(U^\top U)^{-1}(U^\top {A})$, respectively? Definitely, the
answer is no due to the discrete property embedded in
the integer matrix.

\begin{example}\label{ex:rounding}
Consider
%\begin{equation*}
$\min_{\mathbf{v}\in\mathbb{Z}^{k\times 1}}
 \|\mathbf{a} - U\mathbf{v}\|_2 $,
%\end{equation*}
where $U = \left[\begin{array}{cc}8 & 1 \\9 &
2\end{array}\right]$ and $\mathbf{a} = \left[\begin{array}{c}16
\\17\end{array}\right] $. The optimal solution
$\mathbf{v}_{\mbox{opt}}\in\mathbb{R}^{2\times 1}$ can be
obtained by computing
\begin{equation*}
\mathbf{v}_{\mbox{opt}} = (U^\top U)^{-1}(U^\top \mathbf{a})\approx \left[\begin{array}{c}2.1429 \\-1.1429\end{array}\right].
\end{equation*}
%i.e., $\|\mathbf{a} - U\mathbf{v}_{\mbox{opt}}\|_2 = 0$.
Considering integer vectors $\mathbf{v}_1 =
\left[\begin{array}{c}2 \\-1\end{array}\right]$, $\mathbf{v}_2
= \left[\begin{array}{c}2 \\-2\end{array}\right]$,
$\mathbf{v}_3 = \left[\begin{array}{c}3 \\-1\end{array}\right]$
and $\mathbf{v}_4 = \left[\begin{array}{c}3
\\-2\end{array}\right]$ which are around $\mathbf{v}_{\mbox{opt}}$, we
see that $\|\mathbf{a} - U\mathbf{v}_i\|_2  = \sqrt{2},
\sqrt{13}, \sqrt{113}, 6\sqrt{2}$ for $i = 1,\ldots, 4$,
respectively. However, if we take $\mathbf{v} =
\left[\begin{array}{c}2 \\ 0 \end{array}\right]$, we have
$\|\mathbf{a} - U\mathbf{v}\|_2 = 1 < \|\mathbf{a} -
U\mathbf{v}_i\|_2$ for $i = 1,\ldots, 4$. See also
Figure~\ref{roundint} for a demonstration through the geometric
viewpoint.
%Indeed, this result
%can be observed in a direct sense from Figure~\ref{roundint}.
 %
% Fig 3.
%
\begin{figure}[h]
\begin{center}
\epsfig{figure=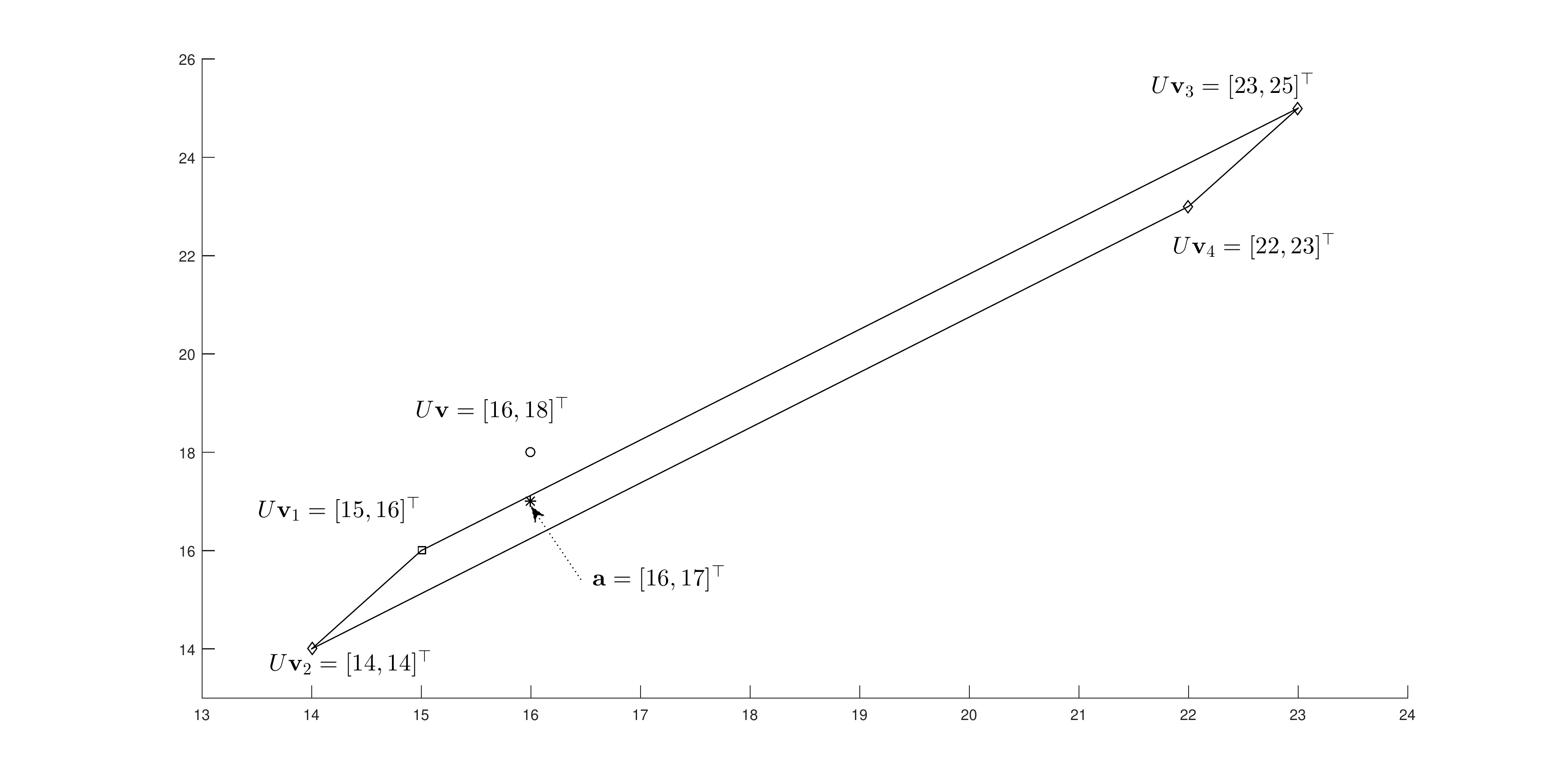, width=13.5cm}
\end{center}
\caption{A counterexample for the rounding approach}
\label{roundint}
\end{figure}

\end{example}

%In summary, the above example gives a glimpse of the complexity of solving the IMF. It is not surprising that the
%challenge lies in computing the global minimum in each iteration.
It follows that the major mechanism of the BCD method lies in computing the global minimum in each iteration.
To this end, we discuss the
integer least squares problem in the next section and apply it
to solve the ILA column-by-column/row-by-row so that a sequence
of descent residuals can be expected.
%Once this problem can be
%tackled, we can enjoy the fruit of our labour. To this end,
%we apply the integer least squares method
%in the next section.
%This approach gives rise to the
%construction of a sequence of descent residuals after searching for the
%optimal solution.

%%%%%%%%%%%%%%%%%%%%%%%%%%%%%%%%%%%%%%%%%%%%%%%%%%%%%%%%%%%%%%%%%%%%%%%%%%%%%%%%%%%%%%%%%%%%
% ILA
%%%%%%%%%%%%%%%%%%%%%%%%%%%%%%%%%%%%%%%%%%%%%%%%%%%%%%%%%%%%%%%%%%%%%%%%%%%%%%%%%%%%%%%%%%%%
\section{Integer Least Squares Problems}\label{sec:ils}
Given a vector $y\in\mathbb{R}^{m}$
and a matrix $H \in\mathbb{R}^{m\times n}$ with full column
rank, the integer least squares (ILS) problem is defined as
\begin{equation}\label{eq:ILS}
\min_{\mathbf{x}\in\mathbb{Z}^{n}} \|\mathbf{y} - H\mathbf{x} \|_2^2.
 \end{equation}
Solving the ILS problem is known to be
NP-hard~\cite{van1981}. In this section, we first review a
common approach for solving ILS problems. This approach is
referred to be the enumeration approach and usually relies on
two major processes, \emph{reduction} and \emph{search}.
Second, we consider the ILS problem with box constraints
(ILSb), which can be defined as:
\begin{equation} \label{eq:ILSb}
\min_{\mathbf{x}\in\mathcal{B}} \|\mathbf{y} - H\mathbf{x} \|_2^2,
 \end{equation}
where $\mathcal{B}
=\mathcal{B}_1\times\cdots\times\mathcal{B}_n $ with
$\mathcal{B}_i = \{x_i\in \mathbb{Z}:
 \ell_i\leq x_i\leq u_i\}$.
%We then enhance this idea to solve \textcolor{red}{ILAB}.
We then enhance this idea of solving the ILSb to solve the ILA with box
constraints, where entries are subject to a bounded region.
Third, we will discuss some convergence properties for the ILA.

\subsection{ILS}
Typical methods for solving ILS problems in the literature have two stages: reduction (or preprocessing) and
search.

\emph{To do the reduction}, the idea is based on the well-known
process, Lenstra-Lenstra-{Lov\'{a}sz}~\cite{Lenstra1982} (LLL)
reduction. This process is to transform the ILS problem defined in~\eqref{eq:ILS}
into a reduced form
\begin{equation}\label{eq:reduced}
\min_{\mathbf{z}\in \mathbb{Z}^n} \|\hat{\mathbf{y}} - R \mathbf{z} \|_2^2,
 \end{equation}
in terms of a specific QRZ factorization~\cite{Chang2009} such that
\begin{eqnarray}\label{eq:qz}
Q^\top H Z &=&
\left[\begin{array}{c}R \\0\end{array}\right], \,\hat{\mathbf{y}} = Q_1^\top \mathbf{y},\,  \mathbf{z} = Z^{-1}\mathbf{x},
\end{eqnarray}
where $Q = \left[\begin{array}{cc} Q_1 & Q_2\end{array} \right]
\in\mathbb{R}^{m\times m}$ is orthogonal,
$Z\in\mathbb{Z}^{n\times n}$ is unimodular (i.e., $\det(Z) =
\pm 1$), and $R = [r_{i,j}]\in\mathbb{R}^{n\times n}$ is an
upper triangular matrix satisfying the following conditions:
\begin{subequations}\label{Rcondition}
 \begin{eqnarray}
 |r_{i-1,j}| &\leq& \dfrac{|r_{i-1,i-1}|}{2},\quad 2\leq i< j\leq n,\label{Rcondition1}\\
 \delta r^2_{i-1,i-1} &\leq& r^2_{i-1,i} + r_{i,i}^2,\quad 2\leq i \leq n,\label{Rcondition2}
 \end{eqnarray}
\end{subequations}
where $\delta$ is a parameter satisfying {$\frac{1}{4}< \delta \leq 1$}. In
this work, we choose $\delta =1$ as {suggested}
in~\cite{Chang2009}. Then it can be easily obtained
from~\eqref{Rcondition} that the diagonal entries of $R$ have
the following properties:
\begin{equation*}
 |r_{i-1,i-1}| \leq \frac{2}{\sqrt{3}}| r_{i,i}|, \quad 2\leq i \leq n.
\end{equation*}
Once having these properties, the matrix $R$ can enhance the
efficiency of a typical search algorithm
for~\eqref{eq:reduced}; see~\cite{Agrell2002, Foschini1999} and
Algorithm~\ref{lllalg} for more details.

\setlength{\algomargin}{1.5 ex}
\restylealgo{algoruled}
\SetKwComment{Comment}{}{}
\begin{algorithm}[t] \linesnumbered
\caption{(LLL reduction)~\cite{Chang2009} \hfill $[R,
Z, \hat{\mathbf{y}}] = \textsc{LLL}(H,\mathbf{y})$}
\label{lllalg} \KwIn{A matrix $H\in\mathbb{R}^{m\times n}$ and a
 vector $y\in\mathbb{R}^m$.}
\KwOut{An upper triangular matrix $R\in \mathbb{R}^{n\times n}$, an
unimodular matrix $Z\in\mathbb{Z}^{n\times n}$, and a vector
$\hat{\mathbf{y}}\in \mathbb{R}^{n}$ for the computation of~\eqref{eq:reduced} and~\eqref{eq:qz}.}
%given in~\eqref{eq:reduced}.}
%\vskip .1in%\Indm %
\Begin{
compute the QR decomposition of $H$, i.e., \\
\Comment{ $H= \left[\begin{array}{cc} Q_1 &
Q_2\end{array}\right] \left[\begin{array}{c}R
\\0\end{array}\right]$\;} set $Z = I_n$,
$\hat{\mathbf{y}}=Q_1^\top \mathbf{y}$ and $k=2$\; \While
{$k\leq n$} {
%    \Comment{/*Size Reduction on $r_{k-1,k}$}
%$R  \leftarrow  RZ_{k-1,k}$,\,
%$Z  \leftarrow ZZ_{k-1,k}$\;
    \Comment{/*Size Reduction on $r_{1:k-1,k}$}
 \For{$ i = k-1,\ldots,1$}
{$R  \leftarrow  RZ_{i,k}$,\,
$Z  \leftarrow ZZ_{i,k}$\;
}
        \eIf{$r_{k-1,k-1}^2 > (r_{k-1,k}^2 + r_{k,k}^2)$}
{
\Comment{/* Permutation and Triangularization}
$R  \leftarrow  G_{k-1,k} R P_{k-1,k}$,\,
$Z  \leftarrow  Z P_{k-1,k}$,\,
$\hat{\mathbf{y}}  \leftarrow  G_{k-1,k} \hat{\mathbf{y}}$\;
\If {$k>2$} {$k  \leftarrow   k-1$\;}
}
{
%  \Comment{/*Size Reduction on $r_{1:k-2,k}$}
%\For{$ i = k-2,\ldots,1$}
%{
%$R  \leftarrow  RZ_{k-1,k}$,\,
%$Z  \leftarrow ZZ_{k-1,k}$\;
%}
%%
$k \leftarrow k+1$\;
}
}
}
\end{algorithm}

Note that the entire QRZ factorization includes a QR
factorization and two types of transformations, integer Gauss
transformations (IGTs) and {permutations}, to
implicitly update $R$ right after the QR factorization of $H$
so that it satisfies~\eqref{Rcondition}. To make this work more
self-contained, we briefly introduce the major results here;
see~\cite{Chang2009} for more details.

First, the IGTs are done by an unimodular matrix $Z_{i,j}$ which is given by
\begin{equation*}
Z_{i,j} = I_n-\zeta_{ij} e_ie_j^\top,
\end{equation*}
where $i\neq j$ and
%$\zeta_{ij}$ is an integer.
%In particular, let
$\zeta_{ij} = \rm{round}(\dfrac{r_{i,j}}{r_{i,i}})$.
 %(if there is a tie, pick up the one with \textcolor{red}{largest} magnitude).
We then multiply $R$ with $Z_{i,j}$ ($i<j$) from the right and
obtain an updated matrix
\begin{equation*}
\widehat{R} \equiv RZ_{i,j}= R-\zeta_{i,j} R e_ie_j^\top.
\end{equation*}
It can be seen that entries of $\widehat{R}$ are equal to
entries of $R$, except that $\hat{r}_{k,j} = r_{k,j}
-\zeta_{i,j} r_{k,i}$ for $k= 1,\ldots, i$, and
satisfies~\eqref{Rcondition1}.

Second, if~\eqref{Rcondition2} is not satisfied for some $i$,
we then swap columns $i-1$ and $i$ of $R$ by a permutation
matrix, denoted by $P_{i-1,i}$. But the resulting matrix $R$ is
no longer upper triangular and is required to be trangularized.
In~\cite{Lenstra1982}, the Gram-Schmidt process (GP) is
applied to bring back the structure. That is, find an
orthogonal matrix $G_{i-1,i}$ so that the resulting matrix
\begin{equation*}
\widehat{\widehat{R}} = G_{i-1,i} \widehat{R} P_{i-1,i}
\end{equation*}
is an upper triangular and satisfies~\eqref{Rcondition1}.

It should be noted that right after these transformations, two
characteristics are worthy of our attention~\cite{Ling2007,
Xie2011}. First, it can be seen that
\begin{equation*}
\min_{\hat{\mathbf{z}}\in \mathbb{Z}^n} \|\hat{\mathbf{y}} - \widehat{R} \hat{\mathbf{z}} \|_2^2
= \min_{\mathbf{z}\in \mathbb{Z}^n} \|\hat{\mathbf{y}} - R \mathbf{z} \|_2^2,
\end{equation*}
since $Z_{i,j}$ is unimodular, that is, the application of the
IGT will not affect the optimal value of~\eqref{eq:reduced}.
Second, it follows from a direct computation that
\begin{subequations}%\label{eq:hat}
  \begin{eqnarray*}
 \hat{\hat{r}}^2_{i-1,i-1} &=&
% ({r}_{i-1,i} - \zeta_{i-1,i}
% {r}_{i-1,i-1})^2
% +r_{i,i}^2
\hat{r}_{i-1,i}^2
 +r_{i,i}^2
 < {r}_{i-1,i}^2+r_{i,i}^2,\\
  |\hat{\hat{r}}_{i-1,i-1}  \hat{\hat{r}}_{i,i}|
  &=& |\det(\widehat{\widehat{R}} )|
  =|\det(R)| =   |r_{i-1,i-1}  r_{i,i}|.%{\color{red}????}
 \end{eqnarray*}
\end{subequations}
This implies that though the IGT will not reduce the optimal
value of~\eqref{eq:reduced}, it will reduce the absolute value
of the original $(i-1,i-1)$ entry and enlarge the absolute
value of the $(i,i)$ entry and hence make a typical search
algorithm more efficient. However, the calculation of the IGTs
would be time-consuming. For example, in
Algorithm~\ref{lllalg}, we need to recompute the IGTs in
previous step, once~\eqref{Rcondition2} is not satisfied. To
prevent from this calculation and enhance the efficiency of the
algorithm, the IGTs are computed only when a permutation is
required. This result can be seen in Algorithm~\ref{plllalg},
which is called the partial LLL reduction
algorithm~\cite{Xie2011}. Specifically, the trangularization
and the initial QR decomposition in~\cite{Lenstra1982} for the
original LLL algorithm are {obtained} through the GP to the permuted matrix $R$ and the initial matrix
$H$. To enhance the stability of Algorithm~\ref{lllalg}, the trangularization
and the initial QR decomposition computed in~\cite{Xie2011} are through Givens rotations
and Householder transformations, respectively.

\setlength{\algomargin}{1.5 ex}
\restylealgo{algoruled}
\SetKwComment{Comment}{}{}
\begin{algorithm}[t] \linesnumbered
\caption{(PLLL reduction)~\cite{Xie2011} \hfill $[R, Z, \hat{\mathbf{y}}] = \textsc{PLLL}(H,\mathbf{y})$}
\label{plllalg}
\KwIn{A matrix $H\in\mathbb{R}^{m\times n}$ and a
 vector $\mathbf{y}\in\mathbb{R}^m$.}
\KwOut{An upper triangular matrix $R\in \mathbb{R}^{n\times n}$, an unimodular matrix $Z\in\mathbb{Z}^{n\times n}$, and a vector $\hat{\mathbf{y}}\in
\mathbb{R}^{n}$  for the computation of~\eqref{eq:reduced} and~\eqref{eq:qz}.}
%given in~\eqref{eq:reduced}.}
%\vskip .1in%\Indm %
\Begin{
compute the QR decomposition of $H$ with minimum column pivoting, i.e., \\
\Comment{ $H=
\left[\begin{array}{cc} Q_1 & Q_2\end{array}\right]
\left[\begin{array}{c}R \\0\end{array}\right]P$\;}
set $Z = P$, $\hat{\mathbf{y}}= Q_1^\top  \mathbf{y}$  and $k=2$\;
\While {$k\leq n$}{
$\zeta= \rm{round}( \dfrac{r_{k-1,k}}{r_{k-1,k-1}})$,\,
$\alpha = (r_{k-1,k} -  \zeta r_{k-1,k-1})$\;
        \eIf{$r_{k-1,k-1}^2 > (\alpha^2 + r_{k,k}^2)$}
   {
   \If{ $\zeta \neq 0$}
   {
\Comment{/*Size Reduction on $r_{k-1,k}$}
{$R  \leftarrow  RZ_{k-1,k}$,\,
$Z  \leftarrow ZZ_{k-1,k}$\;
}
%
%    \Comment{/*Perform a size reduction on $r_{k-1,k}$}
%   $R  \leftarrow  R Z_{k-1,k}$,\,
%$Z  \leftarrow  Z Z_{k-1,k}$,\;
    \Comment{/*Size Reduction on $r_{1:k-2,k}$}
 \For{$ i = k-1,\ldots,1$}
{$R  \leftarrow  RZ_{i,k}$,\,
$Z  \leftarrow ZZ_{i,k}$\;
}
}
%
%
%\If {$k>2$} {$k  \leftarrow   k-1$\;}
\Comment{/* Permutation and Triangularization}
%columns $k-1$ and $k$ of $R$}
$R  \leftarrow  R P_{k-1,k}$,\,
$Z  \leftarrow  Z P_{k-1,k}$\;
%\Comment{/*Bring $R$ back to an upper triangular matrix}
$R  \leftarrow  G_{k-1,k}R$,\,
$\hat{\mathbf{y}}  \leftarrow  G_{k-1,k} \hat{\mathbf{y}}$\;
\If {$k>2$} {$k  \leftarrow   k-1$\;}
}
{
$k \leftarrow k+1$\;
}
}
}

\end{algorithm}

\emph{To do the search}, let us first consider
the following inequality
\begin{equation}\label{eq:hyper}
\|\hat{\mathbf{y}} - R \mathbf{z}\|_2^2 < \beta,
\end{equation}
where $\beta > 0$, $\hat{\mathbf{y}}\in \mathbb{R}^{n}$, $\mathbf{z} = [z_i]\in \mathbb{Z}^n$, and $R
= [r_{i,j}]
\in\mathbb{R}^{n\times n}$.
Note that~\eqref{eq:hyper} gives rise to a hyperellipsoid,
\begin{equation*}
\|\hat{\mathbf{y}} - R \mathbf{z}\|_2^2 =  \beta
\end{equation*}
in terms of variable $\mathbf{z}$. Let $\mathbf{z}$ be a
solution of~\eqref{eq:hyper} and define
\begin{equation*}
c_n = \dfrac{\hat{y}_n}{r_{n,n}},\quad
c_k = \dfrac{\hat{y}_k-\sum_{j=k+1}^n
r_{k,j}z_j}{r_{k,k}},
\end{equation*}
for $k = n-1,\ldots, 1$.
Then~\eqref{eq:hyper} is equivalent to
\begin{equation}\label{eq:r0}
\sum_{i=1}^n r_{i,i}^2 (z_i-c_i)^2 < \beta,
\end{equation}
which implies that
\begin{subequations} \label{eq:r}
\begin{eqnarray}
r_{n,n}^2 (z_n-c_n)^2 &<&\beta, \label{eq:rn}\\
r_{k,k}^2(z_k - c_k)^2 &<& \beta - \sum_{i=k+1}^n r_{i,i}^2(z_i-c_i)^2, \label{eq:rk}
\end{eqnarray}
\end{subequations}
for $k = n-1,\ldots, 1$. Based on the above inequalities, we
can apply the idea of the Schnorr-Euchner (SE) enumerating
strategy to propose the following search process, denoted by
$[\mathbf{z}] =
\textsc{search}(R,\hat{\mathbf{y}})$~\cite{Damen2003}:

\vskip 0.1 in
\begin{enumerate}
\item[Step 1.]
Choose $z_{n} = \rm{round}(c_n)$ and let $k=n-1$.
\item[Step 2.] For $1 < k < n$, consider the following two cases

\begin{itemize}
\item[Step 2a.]  Choose $z_{k} = \rm{round}( c_k )$. Once~\eqref{eq:rk} is satisfied,
move forward to the next level, i.e. search for $z_{k-1}$.

\item[Step 2b.] Otherwise, move back to the previous level, but choose $z_{k+1}$
in Step 2a as the next nearest integer to $c_{k+1}$ with largest magnitude.

\end{itemize}

\item[Step 3.]  Once $k =1$,  consider the following two cases.

\begin{itemize}
\item[Step 3a.]  Choose $z_{1} = \rm{round}(c_1)$, and update
the parameter $\beta$ by defining $\beta = \|\hat{\mathbf{y}} - R \mathbf{z}\|_2^2$.
Move back to Step 2 with $k=2$.

\item[Step 3b.] Otherwise, move back to the previous level, but choose $z_{2}$ in Step 2a as the next nearest integer to $c_{2}$.

\end{itemize}

\item[Step 4.] Once reach the last level, i.e., $k=n$, and~\eqref{eq:rn} is not
satisfied for the current $\beta$, output the latest found integer point $\mathbf{z}$
as the optimal solution of~\eqref{eq:reduced}.
\end{enumerate}
\vskip 0.1 in

Note that to make Step 4 meaningful, the initial bound should
be carefully provided. The simplest way is to assign
infinity as the initial bound. Combing these reduction and
search processes, we then have an approach to
solve~\eqref{eq:reduced}, or equivalently, to
solve~\eqref{eq:ILS}.

\subsection{ILSb}
For the ILSb problem~\eqref{eq:ILSb}, the LLL (PLLL) reduction
cannot be directly applied to obtain the optimal solution. This
is because the unimodular matrix $Z_{i,j}$ obtained in the QRZ
factorization might complicate the box constraints, if
$Z_{i,j}$ is not a permutation matrix. An alternative way to do
the reduction and to enhance the efficiency of search process
is required. It should be noted that till now, the reduction
processes proposed in the literature more or less strive for
diagonal entries arranged in a nondescreasing order, i.e.,
\begin{equation*}
|r_{1,1}|\leq |r_{2,2}|\leq \cdots\leq |r_{n,n}|.
\end{equation*}
This purpose is to reduce the search range of $z_k$, for
$k=n,n-1,\ldots,1$, once the right hand side of~\eqref{eq:r} is
provided.

For the details of the development of the reduction processes, see~\cite{Fincke1985,Damen2003, Wubben2011} and
references therein. It should be noted that the three methods
given in~\cite{Fincke1985,Damen2003, Wubben2011} share a common
weakness, that is, only the information of the matrix $H$ is
used to do the reduction. In~\cite{Chang2008}, Chang and Han
proposed a column reordering approach which uses all available
information, the matrix H and box constraints. We denote this approach as the CH algorithm, which has been shown to be more efficient that the existing algorithms
in~\cite{Fincke1985,Damen2003, Wubben2011}.

The idea of the CH algorithm is to reduce the right hand side
of~\eqref{eq:r}
%, i.e., enlarge each $r_{k,k}^2 (z_k-c_k)^2$ so
%that the search range is reduced.
but not to
arrange diagonal entries in a nondecreasing order directly.
%For this reason, each $z_k$ is selected to be the first nearest
%integer to $c_k$.
This is because even if $|r_{k,k}|$ is very
large, the value of $r_{k,k}^2 (z_k-c_k)^2$ may be very small, and hence
the search range is large. For this reason, they proposed to choose $z_k$
as the second nearest integer to $c_k$, i.e., $|z_k-c_k| >
0.5$. Importantly, once $|r_{k,k}(z_k-c_k)|$ is very large,
$|r_{k,k}|$ is usually very large. We summarize the entire
process as
%This procedure also complies with the requirement
%in~\cite{Damen2003, Wubben2011} and can be summarized as
follows~\cite{Chang2008}:

\vskip 0.1 in
\begin{enumerate}
\item[Step 1.] Compute the QR decomposition of $H$, \begin{equation*}
H=
\left[\begin{array}{cc} Q_1 & Q_2\end{array}\right]
\left[\begin{array}{c}R \\0\end{array}\right],
\end{equation*}
and define $\hat{\mathbf{y}} = Q_1^\top \mathbf{y}$,
$\bar{\mathbf{y}} = \hat{\mathbf{y}}$, and $k=n$.

\item[Step 2.]
    If $k < n$, compute
    $
\bar{\mathbf{y}}(1:k) \leftarrow \bar{\mathbf{y}}(1:k) - R(1:k,k+1)\hat{z}_{k+1}
    $.
Let %$c_k = \bar{y}(k)/r_{k,k}$,
$\widetilde{R} = R(1:k,1:k)$ and $\tilde{\mathbf{y}} =\bar{\mathbf{y}}(1:k)$.
    For $i= 1,\ldots,k$, consider the following two-step process.
\begin{itemize}
\item First, swap columns $i$ and $k$ of $\widetilde{R}$
with a permutation matrix $P_{i,k}$ and return $\widetilde{R}$
to upper triangular with Givens rotations $G_i$, and also use
$G_i$ to update $\tilde{\mathbf{y}}$, that is, we have
\begin{subequations}\label{eq:bar}
\begin{eqnarray}
\overline{R} &\leftarrow & G_i \widetilde{R} P_{i,k},\\
\bar{\mathbf{y}} &\leftarrow & G_i \tilde{\mathbf{y}}.
\end{eqnarray}
\end{subequations}

\item Second,
compute
\begin{equation}\label{eq:zck}
z^{(0)}_{i} = \rm{round}(c_k)_{\mathcal{B}_i}, \, {z}_{i} = \rm{round}(c_k)_{\mathcal{B}_i\backslash z^{(0)}_{i}},
\,\mbox{dist}_i = |{\bar{r}}_{k,k}{z}_{i} - \bar{{y}}_k|,
\end{equation}
where $c_k = \bar{{y}}_k/{\bar{r}}_{k,k}$, $\overline{R} = [{\bar{r}}_{i,j}]$, and $\bar{\mathbf{y}} = [\bar{{y}}_i]$.
\end{itemize}

\item[Step 3.] Let $\mbox{dist}_j = \max_{1\leq i\leq k}
\mbox{dist}_i$. Interchange columns $j$ and $k$ of $P$,
$\mathcal{B}_j$ and $\mathcal{B}_k$, and update $R$ and $\hat{\mathbf{y}}$
by defining
\begin{eqnarray*}
R(1:k,1:k) &\leftarrow& G_j R(1:k,1:k) P_{j,k}, \\
R(1:k,k+1:n) &\leftarrow& G_j R(1:k,k+1:n), \\
\hat{\mathbf{y}}(1:k) &\leftarrow& G_j \hat{\mathbf{y}}(1:k).
\end{eqnarray*}

\item[Step 4.] Let $\hat{z}_k = z_j^{(0)}$. Move back to Step 2 by replacing the index $k$ with $k-1$, unless $k=1$.

\item[Step 5.] Output the upper triangular matrix
$R\in\mathbb{R}^{n\times n}$, the permutation matrix
$P\in\mathbb{Z}^{n\times n}$, the vector $\hat{\mathbf{y}}
\in\mathbf{R}^n$, and the permuted intervals $\mathcal{B}_i$, for
$i=1,\ldots,n$.

\end{enumerate}
\vskip 0.1 in

Here, ``$\rm{round}(c_k)_{\mathcal{B}_i}$" and ``$ \rm{round}(c_k)_{\mathcal{B}_i\backslash z^{(0)}_{i}}$" denote the
closest and second closest integers in $\mathcal{B}_i$ to
$c_k$, respectively. It should be noted that in the CH algorithm,
computing Step 2 is cumbersome. This is because after swapping
columns $i$ and $k$, it requires $k-i$ Givens rotations to
eliminate the last $k-i$ elements in the $i$-th column and
further $k-i-1$ Givens rotations to eliminate the subdiagonal
entries from column $i+1$ to $k$. To simplify the way of doing
permutation and triangularization, Breen and Chang
in~\cite{Breen2012} suggested to rotate $i$-th column to the
$k$-th column and shift columns $i, i+1,\ldots, k$ to the left
one position. This implies that we only require $k-i-1$ Givens
rotations to do the triangularization.

%%
%%
% Boxed-LLL reduction
%%
%
%%
%%
\setlength{\algomargin}{1.5 ex}
\restylealgo{algoruled}
\SetKwComment{Comment}{}{}
\begin{algorithm}[t!!!!!!!!!] \linesnumbered
\caption{(Boxed-LLL reduction) \hfill $[R, Z, \hat{\mathbf{y}}] = \textsc{MCH}(H,\mathbf{y},\mathbf{\ell}, \mathbf{u})$}
\label{blllalg}
\KwIn{A matrix $H\in\mathbb{R}^{m\times n}$, a
 vector $\mathbf{y}\in\mathbb{R}^n$, a lower bound vector $\mathbf{\ell}\in\mathbb{Z}^n$, and an upper bound vector $\mathbf{u}\in\mathbb{Z}^n$.}
\KwOut{An upper triangular matrix $R\in \mathbb{R}^{n\times n}$, an unimodular matrix $Z\in\mathbb{Z}^{n\times n}$, and a vector $\hat{\mathbf{y}}\in
\mathbb{R}^{m}$  for the computation of~\eqref{eq:reduced}  and~\eqref{eq:qz} with updated lower and upper bounds defined by $\mathbf{\ell}\leftarrow Z^\top \mathbf{\ell} $ and
$\mathbf{u}\leftarrow Z^\top \mathbf{u} $, respectively.}
%given in~\eqref{eq:reduced}.}
%\vskip .1in%\Indm %
\Begin{
compute the QR decomposition of $H$, i.e., \\
\Comment{ $H=
\left[\begin{array}{cc} Q_1 & Q_2\end{array}\right]
\left[\begin{array}{c}R \\0\end{array}\right]$\;}
$S \leftarrow R^{-\top}$,\,
 $\hat{\mathbf{y}} \leftarrow Q_1^\top \mathbf{y}$,\,
%$\mathbf{\ell}\leftarrow P^\top \mathbf{\ell} $,\,
%$\mathbf{u}\leftarrow P^\top \mathbf{u} $,\,
$k \leftarrow 2$\;
$\check{R}\leftarrow R$,\, $G\leftarrow I_n$,\,
$Z \leftarrow I_n$,\,
 ${\check{\mathbf{y}}} \leftarrow \hat{\mathbf{y}}$\;
 \For{$k\leftarrow n$ \KwTo $2$}
{
$maxGap\leftarrow -1$\;
\For{$i\leftarrow 1$ \KwTo $k$}
{
 \For{$ i = k-1,\ldots,1$}
{$\alpha \leftarrow \hat{\mathbf{y}}(i:k)^\top S(i:k,i)$, \,
$\mathbf{x}(i) \leftarrow  \rm{round}(\alpha)_{\mathcal{B}_i}$,\,
$\hat{\mathbf{x}}(i) \leftarrow  \rm{round}(\alpha)_{\mathcal{B}_i \setminus x(i)}$,\,
$temp \leftarrow |\alpha - \hat{\mathbf{x}}(i)|/\|S(i:k,i)\|_2$\;
\If {$temp > maxGap$} {
\Comment{/* Define the Babai integer point $\hat{\mathbf{z}.}$}
$maxGap \leftarrow temp$,\,
$j\leftarrow i$\;
%,\,
%$\hat{\mathbf{z}} (k) \leftarrow \hat{\mathbf{x}}(j)$\;
}
}
}
%
%
%\Comment{/* Update matrix Z}
%
%$Z  \leftarrow  Z \widehat{P}_{k,j}$,\,
%\Comment{/* Update and remove the kith entry of $\hat{\mathbf{y}}$}
$\hat{\mathbf{y}}  \leftarrow  \hat{\mathbf{y}}
- R(:,j)x(j)$\;
\If {$j\neq k$} {
\Comment{/* Permutation and Triangularization}
$R  \leftarrow   R \widehat{P}_{k,j} $,\,
$S  \leftarrow  S \widehat{P}_{k,j}$,\,
$Z(1:k,1:k)  \leftarrow  Z(1:k,1:k) \widehat{P}_{k,j}$\;

$R  \leftarrow   \widehat{G}_{k,j} R  $,\,
$S  \leftarrow  \widehat{G}_{k,j} S $,\,
$\hat{\mathbf{y}}  \leftarrow  \widehat{G}_{k,j} \hat{\mathbf{y}}  $,\;
\Comment{/* Update $\ell$, $u$ and $G$}
$\mathbf{\ell}(1:k)\leftarrow  \widehat{P}_{k,j}^\top \mathbf{\ell}(1:k) $,\,
$\mathbf{u}(1:k)\leftarrow  \widehat{P}_{k,j}^\top \mathbf{u} (1:k)$,\,
$G(1:k,1:k)\leftarrow \widehat{G}_{k,j} G(1:k,1:k)$
\;
}
\Comment{
/* Remove the last rows of $R$, $S$ and $\hat{\mathbf{y}}$, and the last columns of $R$ and $S$.}
$\hat{\mathbf{y}} \leftarrow \hat{\mathbf{y}}(1:k-1)$,\,
$R(:,k) \leftarrow []$,\,  $S(:,k) \leftarrow []$,\,
$R(k,:) \leftarrow []$,\,  $S(k,:) \leftarrow []$\;
}

$R \leftarrow G \check{R} Z$,\, %$Z \leftarrow PZ$,\,
$\hat{\mathbf{y}} \leftarrow G \check{\mathbf{y}}$\; }
\end{algorithm}

Later, Su and Wassell proposed a geometric approach to
efficiently compute $c_k$, provided that $H$ is
nonsingular~\cite{Su2005}. Indeed,
the matrix $H$ are only required to have full column rank and could
be non-square, since, from~\eqref{eq:bar} and~\eqref{eq:zck}, we have
\begin{subequations}\label{eq:simple}
\begin{eqnarray}
z_i^{(0)} &=&  \rm{round}(c_k)_{\mathcal{B}_i}
=  \rm{round}(e_k^\top \overline{R}^{-1} \bar{\mathbf{y}})_{\mathcal{B}_i}
%=  \lfloor e_k^\top
% P_{i,k}^\top \widetilde{R}^{-\top} G_i^\top
% G_i\tilde{\mathbf{y}}\rceil_{\mathcal{B}_i}
 = \rm{round}(e_i^\top
 \widetilde{R}^{-1} \tilde{\mathbf{y}})_{\mathcal{B}_i},\\
\mbox{dist}_i &=&  |{\bar{r}}_{k,k}{z}_{i} - \bar{{y}}_k|
=  |{\bar{r}}_{k,k}||{z}_{i} - \dfrac{\bar{{y}}_k}{{\bar{r}}_{k,k}}|
= \dfrac{|z_i-c_k|}{\|\overline{R}^{-\top} e_k\|_2}
= \dfrac{|z_i-c_k|}{\|\widetilde{R}^{-\top} e_i\|_2}.
\end{eqnarray}
\end{subequations}
This implies that we can simplify the CH algorithm in terms of the
formulae given in~\eqref{eq:simple}. A similar, but
complicated, discussion can be found in~\cite{Breen2012}.
Additionally, the algorithm given in~\cite{Breen2012} focuses
more on how to do the column reordering without deriving the
final upper triangular matrix $R$, a required information for
solving ILSb. Here, we strengthen the algorithm and provide the
details in Algorithm~\ref{blllalg}. We must emphasize that the computed $c_k$ in~\cite{Breen2012} is wrongly defined to be
$e_i^\top
 \widetilde{R}^{-1} {\mathbf{y}}$, but the correct expression should be $e_i^\top
 \widetilde{R}^{-1} \tilde{\mathbf{y}}$ (see line 8 of Algorithm 3
in~\cite{Breen2012}).
%Second, once $j$ has been determined,
%they simply removed the last entry of $\hat{\mathbf{y}}$ (see
%line 20 of the same Algorithm). This will lead to an incorrect
%multiplication with the Given rotations later.

To solve the ILSb problems, the search process has to take the
box constraint into account. That is, during the entire search
process, we have to search for an integer vector
satisfying~\eqref{eq:r} and the box constraint
in~\eqref{eq:ILSb} simultaneously. Though this constraint
complicate our search process, we can allow this complexity to
shrink the search range. This observation has been given by
Chang and Han in~\cite{Chang2008}. Here, we summarize their
result as follows. Notice that since $z_i \in \mathcal{B}_i$
for each $i$, we have
\begin{equation}\label{eq:low}
\hat{y}_k - \sum_{i = k}^n \max(r_{k,i} \ell_i, r_{k,i} u_i)\leq
\hat{y}_k - \sum_{i = k}^n r_{k,i}z_i\leq
\hat{y}_k - \sum_{i = k}^n \min(r_{k,i} \ell_i, r_{k,i} u_i),
\end{equation}
It follows that
\begin{equation*}
 (\hat{y}_k - \sum_{i = k}^n r_{k,i}z_i)^2 \geq \delta_k.
\end{equation*}
Here, we define
\begin{equation*}
\delta_k = \min\{
(\hat{y}_k - \sum_{i = k}^n \max(r_{k,i} \ell_i, r_{k,i} u_i))^2,
(\hat{y}_k - \sum_{i = k}^n \min(r_{k,i} \ell_i, r_{k,i} u_i))^2
\},
\end{equation*}
if the upper and lower bounds in~\eqref{eq:low} have the same sign; otherwise, take $\delta_k = 0$.
Let
\begin{eqnarray*}
t_n &=& 0, \quad t_k=\sum_{i={k+1}}^n r_{i,i}^2(z_i - c_i)^2,\quad k=1,\ldots,n-1,\\
\gamma_1 &=& 0, \quad \gamma_k=\sum_{i=1}^{k-1} \delta_i,\quad k=2,\ldots,n.
\end{eqnarray*}
Since $(\hat{y}_k - \sum_{j=k}^n r_{k,j} z_j)^2 =
r_{k,k}^2(z_k-c_k)^2$,~\eqref{eq:r0} implies that
\begin{equation*}
\beta > \sum_{i=1}^n r_{i,i}^2 (z_i-c_i)^2 =
\sum_{i=1}^n (\hat{y}_i - \sum_{j=i}^n r_{i,j} z_j)^2
\geq
\gamma_k + r_{k,k}^2(z_k-c_k)^2+
t_k,
\end{equation*}
that is,
\begin{equation}\label{eq:reducedB}
r_{k,k}^2(z_k-c_k)^2 < \beta -
\gamma_k  - t_k.
\end{equation}

We thus obtain an upper bound which is at least as tight as
that in~\eqref{eq:r}. Upon using this bound, we have the
following search process given in~\cite{Chang2008}.

\vskip 0.1 in
\begin{enumerate}

%\item[Step 0.] Compute $\gamma_i$ for $i=1,\ldots, n$.

\item[Step 1.]
Choose $z_{n} = \rm{round}(c_n)_{\mathcal{B}_n}$ and let $k=n-1$.
\item[Step 2.] For $1 < k < n$, consider the following two cases.
%choose $\zeta_{k} = \lfloor c_k \rceil_{\mathcal{B}_k}$.  Once~\eqref{eq:reducedB} is not satisfied,
%move back to the previous step, $k+1$, and choose $z_{k+1}$ as
%the next nearest integer in $\mathcal{B}_{k+1}$ to $c_{k+1}$;
%otherwise move forward to the next step, $k-1$.

\begin{itemize}
\item[Step 2a] Choose $z_{k} = \rm{round}(c_k)_{\mathcal{B}_k}$.  Once~\eqref{eq:reducedB}
is satisfied, move forward to the next step, $k-1$.

\item[Step 2b] Otherwise, move back to the previous step, $k+1$, but choose $z_{k+1}$ in Step 2a
as the next nearest integer in $\mathcal{B}_{k+1}$ to $c_{k+1}$
until~\eqref{eq:reducedB} is satisfied. If all integers in
$\mathcal{B}_{k+1}$ have been chosen, move further back to
$k+2$ and so on.

\end{itemize}

%\item[Step 3.]  For $k =1$, choose a valid integer for $z_1$
%to satisfy~\eqref{eq:rk}, if it is possible, and update the parameter $\beta$ by defining
%$\beta = \|\hat{\mathbf{y}} - R \mathbf{z}\|_2^2$.  Move back to Step 2 with $k=2$.

\item[Step 3.] Once $k =1$, consider the following two cases.

\begin{itemize}
\item[Step 3a.] Choose $z_{1} = \rm{round}(c_k)_{\mathcal{B}_1}$, and update the
parameter $\beta$ by defining $\beta = \|\hat{\mathbf{y}} - R
\mathbf{z}\|_2^2$.  Move back to Step 2 with $k=2$.

\item[Step 3b.] Otherwise, move back to the previous level, but choose $z_{2}$ in Step 2a
as the next nearest integer in $\mathcal{B}_{2}$ to $c_{2}$
until~\eqref{eq:reducedB} is satisfied.

\end{itemize}

\item[Step 4.] Once reach the last level, i.e., $k=n$, and~\eqref{eq:reducedB}
is not satisfied for the current $\beta$, output the latest
found integer point $\mathbf{z}$ as the optimal solution
of~\eqref{eq:reduced}.
\end{enumerate}
\vskip 0.1 in

For this search process, two issues deserve our attention. First, in Step 2b, the strategy to move back
to the previous step must be continued until the iterations
reach the last step or a step which has not been completely
searched. For example, once the iteration moves back to the
previous step, say step $k+1$, it might be the case that entries in steps $k+1$ and
$k+2$ have been completely searched. Therefore, the iteration
should move further back to $k+3$, instead of updating
$z_{k+2}$ soon after moving back. Indeed, this phenomenon
seems to be ignored in step 5) of the search
algorithm~\cite{Chang2008} and may lead to an infinity loop in
some case. Second, like the search process for the ILS, the
initial bound $\beta$ is set to be $\infty$ so that Step 1 is
accessible from the beginning.
%Note that in the above algorithm we assume without loss of
%generality that the initial bound $\beta = \infty$ so that Step
%1 accessible at the beginning.

Based on the approaches for solving ILS and ILSb problems, we
now have an BCD approach to solve the ILA with or without box
constraints. Without elaborating on both cases, we summarize a
procedure for solving the ILA in Algorithm~\ref{alg:ILA},
provided with the initial matrix $V$. The similar discussion
can be generalized to the ILA with box constraints by
replacing the reduction and search processes in terms of the
ILSb approach.
\setlength{\algomargin}{1.5 ex}
\restylealgo{algoruled}
\SetKwComment{Comment}{}{}
\begin{algorithm}[h!!!!] \linesnumbered
\caption{(ILA) \hfill $[U, V] = \textsc{ila}(A,V)$}
\label{alg:ILA}
\KwIn{A matrix $A\in\mathbb{Z}^{m\times n}$ and an initial matrix
 $V\in\mathbb{Z}^{k\times n}$ for the computation of~\eqref{constraint1}.}
\KwOut{$U\in \mathbb{Z}^{m\times k}$ and $V\in\mathbb{Z}^{k\times n}$ as a minimizer of~\eqref{eq:ila}.}
\Begin{
\Repeat{a convergence is attained}{
\Comment{/* Update $U$.}
\For{$ i = 1,\ldots,m$}
{ $[R, Z, \hat{\mathbf{y}}] = \textsc{PLLL}(V^\top,A(i,:)^\top)$\;
$[\mathbf{z}] = \textsc{search}(R,\hat{\mathbf{y}})$\;
$U(i,:) \leftarrow (Z\mathbf{z})^\top$\;
}
\Comment{/* Update $V$.}
\For{$ j = 1,\ldots,n$}
{ $[R, Z, \hat{\mathbf{y}}] = \textsc{PLLL}(U,A(:,j))$\;
$[\mathbf{z}] = \textsc{search}(R,\hat{\mathbf{y}})$\;
$V(:,j) \leftarrow Z\mathbf{z}$\;
}
}
}
\end{algorithm}

\subsection{Properties of Convergence}
Now, we have a BCD approach to solve the ILA.
%Our next goal is to
%investigate the properties under which the BCD approach
%converges.
%
Let $\{(U_i,V_i)\}$ be a sequence of optimal matrices
obtained from the computation of~\eqref{constraint} via Algorithm~\ref{alg:ILA}.
In this section, we would like to show that the sequence  $\{\|A-U_iV_i\|_F\}$ is non-increasing
and
\begin{equation}\label{eq:convergence}
{\lim_{i\rightarrow\infty}\|A - U_i V_i\|_F=\|A - UV\|_F}
\end{equation}
for some particular $U\in\mathbb{Z}^{m\times k}$ and $V\in\mathbb{Z}^{k\times n}$.

To prove that the sequence  $\{\|A-U_iV_i\|_F\}$ is non-increasing, we show that
 \begin{equation*}
\|A - U_i V_i\|_F \geq \|A- U_{i+1} V_i\|_F \geq \|A - U_{i+1}
V_{i+1}\|_F
\end{equation*}
for each $i$. This amounts to showing that Algorithm~\ref{alg:ILA} computes the
optimal solution of
\begin{equation*}
\min_{u\in\mathbb{Z}^{m\times k}} \|A(:,j)-
UV(:,j)\|_2
\end{equation*}
for each $j$, and the optimal solution of
\begin{equation*}
\min_{v\in\mathbb{Z}^{k\times n}} \|A(i,:) - U(i,:)v\|_2
\end{equation*}
 for each
$i$ and can be written as follows.

\begin{theorem}\label{thm:decreasing}
The two processes, the LLL reduction {(Boxed-LLL
reduction)} and the SE search strategy {(SE search
strategy with box constraints)}, discussed in Section~\ref{sec:ils}
provide a global {optimal solution to} the ILS
{(ILSb)} problems.
\end{theorem}
\begin{proof}
%Without loss of generality, let the unimodular matrix $Z$ in
%Algorithm~\ref{plllalg} (respectively, the permutation matrix $Z$ in
%Algorithm~\ref{blllalg}) be an identity matrix.
We use the
notations in Section~\ref{sec:ils} and let $\hat{\mathbf{z}}$
be the vector obtained by the above reduction and search
processes. If there exists an integer vector
$\mathbf{z}\in\mathbb{Z}^{n}$ (respectively,
$\mathbf{z}\in\mathcal{B}$) satisfying
\begin{equation*}
 \|\hat{\mathbf{y}} - R \mathbf{z} \|_2^2 < \|\hat{\mathbf{y}} - R \mathbf{\hat{z}} \|_2^2,%= \|\mathbf{y} - H \mathbf{z} \|_2^2,
\end{equation*}
then
$(\hat{y}_n -r_{n,n}z_n )^2 < \|\hat{\mathbf{y}} - R \mathbf{\hat{z}} \|_2^2$, which contradicts the search strategy.
\end{proof}

%In the above theorem,
%
%In the above theorem, we see that
%
%Using the above theorem, we know that~\eqref{eq:decreasing} holds as $U_2$ is the best solution for the \textcolor{red}{minimization problem}
%\begin{equation}
%\min_{U\in \mathbb{Z}^{m\times k}} \|A - U V_1\|_2
%\end{equation}
%and $V_2$ is the best solution for the \textcolor{red}{minimization} problem
%\begin{equation}
%\min_{V\in \mathbb{Z}^{k\times n}} \|A - U_2 V\|_2.
%\end{equation}
We remark that the above theorem indicates that if $A = UV$ and
the initial value $V_0$ of~\eqref{constraint1} equal to $V$,
then after one iteration, we have
$$\|A - U_1 V_0\|_F = 0.$$
{The same result holds if we initialize $U_0=U$ and
iterate~\eqref{constraint2} first.} Also,
Theorem~\ref{thm:decreasing} implies that $\{\|A - U_iV_i\|_F\}$ is
a non-increasing sequence and, hence, converges. The remaining issue
is whether the sequence $\{U_i,V_i\}$ has at least one limit point. %,
%i.e., the existence of a convergent subsequence such that for some
%$U\in\mathbb{Z}^{m\times k}$ and $V\in\mathbb{Z}^{k\times n}$,
%\begin{equation}\label{eq:convergence}
%\|A - UV\|_F = \lim_{i\rightarrow\infty}\|A - U_i V_i\|_F
%\end{equation}
%is satisfied.
%
%the sequence $\{U_i, V_i\}$ will converge to a solution. However,
%if we further entries of $U_i$'s and $V_i$'s to a bounded
%domain and solve this problem in terms of the ILSb approach,
%the following convergence property can be derived.
In the optimization analysis, this property is not necessary true
unless a further constraint such as the boundedness of the
feasible region is added.
\begin{theorem}~\label{thm:convergence}
Suppose that $U_i$'s  and $V_i$'s are two alternating results
for~\eqref{constraint}. If $U_i$'s and $V_i$'s are
limited to a bounded region $\mathcal{B}_U$ and $\mathcal{B}_V$ of
subsets of $ \mathbb{Z}^{m\times k}$ and $\mathbb{Z}^{k\times n}$,
respectively, then ~\eqref{eq:convergence} holds for some
$U\in\mathcal{B}_U$ and $V\in\mathcal{B}_V$.
\end{theorem}

\begin{proof}
Since the set $\mathcal{B}_U$ is compact,
$\{U_i\}$ has a convergent subsequence, say  $\{U_{i_k}\}$.
This implies that
\begin{equation*}
\lim\limits_{k\rightarrow \infty} U_{i_k} = U
\end{equation*}
for some $U\in\mathbb{Z}^{m\times k}$.
Similarly, the set $\mathcal{B}_V$ is compact. Thus, there
exists a subsequence $\{V_{i_{k_\ell}}\}$of $\{V_{i_k}\}$ such that
$
\lim\limits_{\ell\rightarrow \infty} V_{i_{k_\ell}} = V,
$
for some $V\in\mathbb{Z}^{k\times n}$, which implies
\begin{equation*}
\lim\limits_{i\rightarrow \infty} \|A - U_iV_i\|_F  = \lim\limits_{\ell\rightarrow \infty} \|A - U_{i_{k_\ell}}V_{i_{k_\ell}}\|_F = \|A-UV\|_F.
\end{equation*}
This proves the theorem.
%The last result follows from
\end{proof}

Note that Theorem~\ref{thm:convergence} also facilitates a
pleasant interpretation on the result of convergence. That is,
once the limit of~\eqref{thm:convergence} equals to zero, every
limit point of the subsequence of $\{U_i,V_i\}$ is an exact
decomposition of the original matrix $A$. At
this particular moment, say $(U_k, V_k)$, we have $\|A - U_k
V_k\|_F = 0$ with $U_k$ and $V_k$ satisfying the restricted
bounded region. Regarding this, we should emphasize that even if the sequences $U_i$'s and $V_i$'s converge to some particular $U$ and $V$,
it does not mean that we obtain a local/global minimum for the objective function~\eqref{eq:ila}. This is because what we consider are discrete data sets. It would be hard to define the local minimum as is used in continuous data sets and worthy of our further investigation.

\section{Numerical Experiments} \label{sec:test}
In this section, we carry out three experiments on integer data
sets. In the first one, we want to illustrate the capacity of our
algorithms to do the association rule mining, and in the second one,
we randomly generate integer data sets and assess the low rank
approximation in terms of the ILA approaches with different initial
values. In the third case, we compare our results with the SVD and
NMF approaches by rounding the obtained approximations.
Particularly, in our experiments, we take square roots of the
desired singular values and assign them to the corresponding left
and right singular vectors before doing the rounding approach for
the SVD.

\subsection{Association Analysis}
Association rule learning is a well-studied method for discovering
embedded relations between variables in a given data set. For
example, in Table~\ref{trandata}, we want to predict whether a
customer who buy two diapers and one egg will continue to buy a
beer. Since each shopping item is inseparable, we record customer shopping behavior in a discrete system. Conventional approaches to
analyze this data are through a Boolean expression~\cite{KG03,KGR06}
so that a rule such as ``$\{\mathrm{diaper, egg}\} \Rightarrow
\{\mathrm{beer}\}$" can be found in the sales data, but how the
quantity affects the marketing activities cannot be revealed. Thus,
we want to demonstrate how the ILA can be applied to do the
association rule learning with quantity analysis. We use the toy
data set given in Table~\ref{trandata} as an example. Definitely,
the same idea can be applied to an extended file of applications,
including Web usage mining, cheminformatics, and intrusion detection
with large data sets.

To begin with, we pick up two most frequent entries in each column
of $A$ as the initial input matrix, for instance,
\begin{equation*} V = \left[
  \begin{array}{cccccc}
   2&1&2&0&2&4\\
   0&0&1&1&0&1\\
  \end{array}
  \right]
\end{equation*}
so that the rows of the matrix can be decomposed according to
the pattern that occurs more frequently in $A$. {Upon using
Algorithm \ref{alg:ILA} without and with  box constraints
$0\leq U^{(2)}_{i,j}\leq 2, 0\leq V^{(2)}_{i,j}\leq 4$,
respectively, we can get the following two approximations:
\begin{eqnarray*}%\label{tranLRF}
  U^{(1)} &=& \left[\begin{array}{cc}
    1&1\\1&0\\0&3\\2&-1\\0&1\\
  \end{array}\right],\quad
V^{(1)} = \left[\begin{array}{cccccc}
    2&1&2&0&2&4\\0&0&1&1&0&1\\
  \end{array}\right]\\
U^{(2)} &=& \left[\begin{array}{cc}
    1&1\\1&0\\0&2\\2&0\\0&1\\
  \end{array}\right],\quad
V^{(2)} = \left[\begin{array}{cccccc}
    2&1&1&0&2&4\\0&0&2&1&0&1
  \end{array}\right]
 \end{eqnarray*}
with  the residual $\|A-U^{(i)}V^{(i)}\|_F^2=9$ and $1$, for
$i=1$ and $2$, respectively.
%
%While applying built-in ``svd" and
%``nnmf" solvers in MATLAB with a rank-2 approximation, we have the approximations:
% \begin{eqnarray*}%\label{tranLRF}
% U^{(3)} &=& \left[\begin{array}{cc}
%    -2&0\\-1&0\\-1&-2\\-3&1\\0&-1\\
%  \end{array}\right],\quad
%%
%V^{(3)} = \left[\begin{array}{cccccc}
%    -1&-1&-1&0&-1&-3\\1&0&-2&-1&1&0
%  \end{array}\right],\\
%%
% U^{(4)} &=& \left[\begin{array}{cc}
%    5&2\\5&0\\0&5\\10&0\\0&2\\
%  \end{array}\right],\quad
%%
%V^{(4)} = \left[\begin{array}{cccccc}
%    0&0&0&0&0&1\\0&0&1&0&0&0\\
%  \end{array}\right],
% \end{eqnarray*}
%with residuals $\|A-U^{(i)}V^{(i)}\|_F^2=18$ and $76$, for
%$i=3$ and $4$, respectively. While solving the ILA, the
%overhead results seem to suggest that our ILA methods perform
%much better than the SVD and NMF approaches.
In this case, we can see that we not only obtain the best
approximation, but also makes the obtained results more
interpretable than those approximated by unconstrained optimization
techniques. Furthermore, from the decomposition, we can obtain more
useful information, such as ``$\{\mbox{2 bags of bread, 1 milk, 1
diaper}\} \Rightarrow \{\mbox{2 chips, 4 beers}\}$". However, like
the conventional BCD approaches, the final result of the ILA depends
highly on the initial values. Our next example is to assess the
performance of the ILA approaches under different initial values.

%%%%%%%%%%%%%%%%%%%%%%%%%%%%%%%%%%%%%%%%%%%%%%%%%%%%%%%%%%%%%%%%%%%%%%%%%%%%%%%%%%%%%%%%%%%%

\subsection{Random test}
%In this example, we want to describe the dependence of our method on
%the initial point.
Let
\begin{eqnarray*}
  A=\left[\begin{array}{ccccc}
    16 &  9 &  7 & 12 & 13\\
    20 & 12 &  8 & 14 & 14\\
    22 & 12 & 10 & 17 & 19\\
    22 & 14 & 10 & 16 & 17\\
    28 & 17 & 13 & 21 & 23
  \end{array}\right]=\left[\begin{array}{ccc}
    2 & 2 & 1\\ 2 & 2 & 2\\ 3 & 3 & 1\\ 2 & 3 & 2\\ 3 & 4 & 2
  \end{array}\right]
  \left[\begin{array}{ccccc}
    4 & 1 & 1 & 3 & 3\\
     2 & 2 & 2 & 2 & 3\\
     4 & 3 & 1 & 2 & 1
  \end{array}\right].
\end{eqnarray*}
Upon using the IMF with constraints $1\leq U_{i,j},V_{i,j}\leq 4$
and the initial matrix
\begin{eqnarray*}
  V_0=\left[\begin{array}{ccccc}
    3 & 2 & 4 & 3 & 3\\
    2 & 1 & 3 & 3 & 4\\
    2 & 2 & 3 & 4 & 1
  \end{array}\right],
\end{eqnarray*}
we have the computed solution
\begin{eqnarray*}
  U^*=\left[\begin{array}{ccc}
    1 & 2 & 1\\3 & 1 & 1\\2 & 3 & 1\\3 & 1 & 1\\4 & 2 & 1
  \end{array}\right],
  V^*=\left[\begin{array}{ccccc}
    4 & 3 & 2 & 3 & 3\\
    4 & 1 & 2 & 3 & 3\\
    4 & 3 & 1 & 3 & 4
  \end{array}\right],
\end{eqnarray*}
and $\|A-U^*V^*\|_F^2=23$.

However, if the initial
matrix
\begin{eqnarray*}
  V_0=\left[\begin{array}{ccccc}
    2 & 3 & 2 & 4 & 1\\
    3 & 2 & 2 & 1 & 2\\
    2 & 1 & 4 & 3 & 3
  \end{array}\right],
\end{eqnarray*}
the computed solution is
\begin{eqnarray*}
  U^*=\left[\begin{array}{ccc}
    2 & 2 & 1\\1 & 4 & 1\\3 & 3 & 1\\2 & 4 & 1\\3 & 4 & 2
  \end{array}\right],
  V^*=\left[\begin{array}{ccccc}
    3 & 1 & 1 & 3 & 3\\
    3 & 2 & 1 & 2 & 2\\
    4 & 3 & 3 & 2 & 3
  \end{array}\right],
\end{eqnarray*}
and $\|A-U^*V^*\|_F^2=7$. Definitely, this is a totally different
result. In Theorem \ref{thm:convergence}, we see that the sequence
$\{(U_k, V_k)\}$ is convergent. However, we can not guarantee that
the convergent sequence provides the optimal solution
of~\eqref{eq:ila}.

In Figure~\ref{distribution}, the distribution
of the residual $\|A-U^*V^*\|_F^2$ obtained by randomly
choosing $100$ initial points is shown.
\begin{figure}[h]
\begin{center}
\epsfig{figure=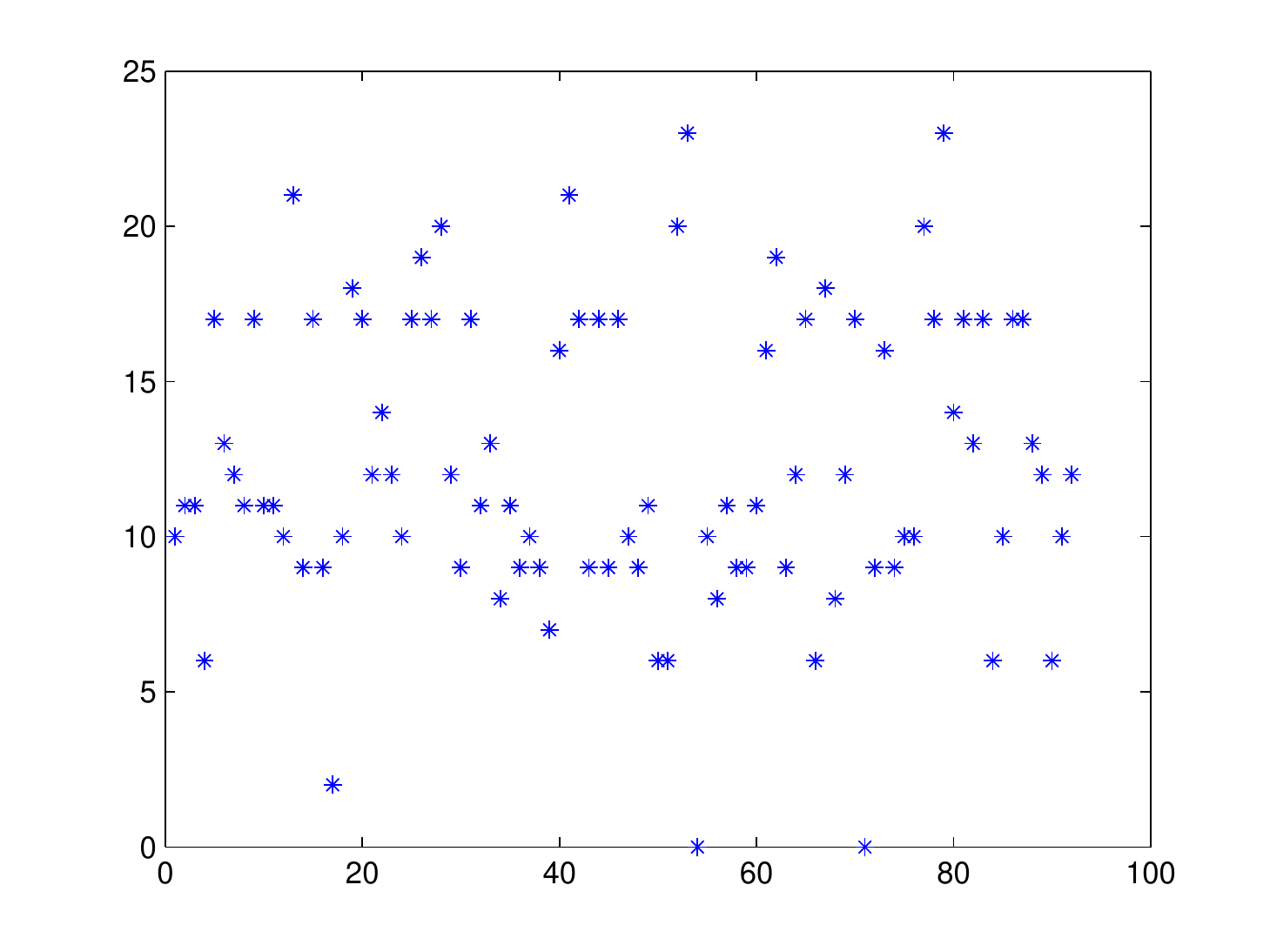, height=8cm,width=12.5cm}
\end{center}
\caption{Distribution of the residuals} \label{distribution}
\end{figure}
We can see from Figure \ref{distribution} that though almost all
residuals are located in the interval $[5,20]$, the zero residual
indicates that a good initial matrix can exactly recover the matrix
$A$. Our algorithm is designed for matrix with full
column rank, therefore, in the iterative process, if $U$ (or
$V^\top$) is not full column rank, we have to stop the iteration.
The number of failures is 7, therefore, there are only 93 points in
Figure \ref{distribution}.

In our next example, we want to show our algorithm
can find a more accurate solution, while comparing it with the
existing methods.
\subsection{Comparison of methods}
For a
given matrix
\begin{eqnarray*}
  A=UV
\end{eqnarray*}
where $U\in\mathbb{Z}^{n\times r}$ and $V\in\mathbb{Z}^{r\times n}$
are two random integer matrices, and $1\leq U_{i,j},V_{i,j}\leq
4$, we want to find a low-rank approximation
\begin{eqnarray*}
  A\approx W H,
\end{eqnarray*}
where $W\in\mathbb{Z}^{n\times r},H\in\mathbb{Z}^{r\times n}$.
Ideally, we want to obtain the solution $W=U$ and $H=V$,
however, we know this ideal result can hardly be achieved due
to the choice of the initial point.

In the following table, the values in the column ``ILSb'' are the
results obtained by our ILSb method, the values in the columns
``rSVD'' and ``rNMF'' are the results obtained by rounding the real
solutions by the MATLAB built-in commands {\it svd} and {\it nnmf}.
The interval (``Interval'') that contains all residual values and
the average residual value(``Aver.'') obtained by randomly choosing
100 initial points, the average number of iterations (``it.
$\sharp$'') by our algorithm, and the percentage of our methods
superior to the existing methods with the same initial value
(``Percent'') are also presented. Similar to the example
in Section 4.2, when $U$ (or $V^\top$) is not full column rank, we
stop the iteration. The number of failures is then recorded in the
column ``fail''.
\begin{table}[h!!!]
  \caption{Performance of different algorithms for fixed $r=n/5$}\label{table:com}
\begin{center}
  \begin{tabular}{ccccccccc}\hline
   \multirow{2}{*} n & \multicolumn{4}{c}{ILSb} & \multirow{2}{*}{rSVD} & \multicolumn{2}{c}{rNMF}
   & \multirow{2}{*}{Percent} \\
   \cline{2-5}\cline{7-8}  & it. $\sharp$ & Interval & Aver. & fail &  & Interval & Aver. &  \\\hline
20 & 6.4 & [0,723] & 490.52 & 0 & 1602 & [229067,242951] & 242537.25 & 100$\%$ \\
30 & 9.3 & [596,1421] & 1118.23 & 2 & 5482 & [1145273,1175381] & 1174450.61 & 100$\%$ \\
40 & 12.4 & [1891,3079] & 2535.03 & 14 & 17345 & [4229094,4300890] & 4299025.57 & 100$\%$\\
50 & 15.6 & [3168,4492] & 3761.25  & 18 & 32245 & [9896777,9999839] & 9994654.52 & 100$\%$\\
  \end{tabular}
  \end{center}
\end{table}

Note that the average rounding residual values by the SVD and NMF are much larger than the ones by our ILSb approach,
or even larger than the worst residual value in our
experiments. This phenomenon strongly shows that while handling
integer matrix factorization, in particular, with box
constraints, our method is more accurate than the convectional
methods.

%%
%%%%%%%%%%%%%%%%%%%%%%%%%%%%%%%%%%%%%%%%%%%%%%%%%%%%%%%%%%%%%%%%%%%%%%%%%%%%%%%%%%%%%%%%%%%%%%%%%%%%%%%
%%% Conclusions                                                                                      %%
%%%%%%%%%%%%%%%%%%%%%%%%%%%%%%%%%%%%%%%%%%%%%%%%%%%%%%%%%%%%%%%%%%%%%%%%%%%%%%%%%%%%%%%%%%%%%%%%%%%%%%%
%
\section{Conclusions}\label{sec:conclusion}
Matrix factorization has long been an important technique in
data analysis due to its capacity of extracting useful
information, providing decision-making, and drawing a
conclusion from a given data set. Conventional techniques
developed so far focus more on continuous data sets with real
or nonnegative entries and cannot be directly applied to handle
discrete data sets. Based on the ILS and ILSb techniques, the main
contribution of this work is to offer an effectual approach to
examine in detail the constitution or structure of a discrete
information. Numerical experiments seem to suggest that our ILA approach works very well in low rank
approximation to integer data sets.

Note that our approach is based on the column-by-column/row-by-row approximation. The calculation of each column/row is independent of each other. This implies that the parallel computation, as is applied in~\cite{Kannan2012}, can be utilized to speed up our calculation, while analyzing a large scale data matrix.

\section*{Acknowledgement}
This first author's research was supported in part by
the National Natural Science Foundation of China under grant
11101067 and the Fundamental Research Funds for the Central
Universities.
The second author's research was supported in part by the National Science Council of Taiwan
under grant 101-2115-M-194-007-MY3.
The third author's research was supported in part by the
 Defense Advanced Research Projects Agency (DARPA)
 XDATA program grant FA8750-
12-2-0309 and NSF grants CCF-0808863, IIS-1242304, and IIS-
1231742. Any opinions, findings and conclusions or recommendations
expressed in this material are those of the authors and do not
necessarily reflect the views of the funding agencies.

\bibliographystyle{unsrt}
%\bibliography{imf}

\begin{thebibliography}{10}

\bibitem{Agrell2002}
E.~Agrell, T.~Eriksson, A.~Vardy, and K.~Zeger.
\newblock Closest point search in lattices.
\newblock {\em Information Theory, IEEE Transactions on}, 48(8):2201--2214, Aug
  2002.

\bibitem{Hassibi1996}
A.~Hassibi and S.~Boyd.
\newblock Integer parameter estimation in linear models with applications to
  gps.
\newblock In {\em Decision and Control, 1996., Proceedings of the 35th IEEE
  Conference on}, volume~3, pages 3245--3251 vol.3, Dec 1996.

\bibitem{Morgan01}
Shona~D. Morgan.
\newblock Cluster analysis in electronic manufacturing.
\newblock {\em Ph.D. dissertation, North Carolina State University, Raleigh, NC
  27695.}, 2001.

\bibitem{Lin2011}
Matthew~M. Lin.
\newblock Discrete eckart-young theorem for integer matrices.
\newblock {\em SIAM Journal on Matrix Analysis and Applications},
  32(4):1367--1382, 2011.

\bibitem{KG03}
Mehmet Koyut\"{u}rk and Ananth Grama.
\newblock Proximus: a framework for analyzing very high dimensional
  discrete-attributed datasets.
\newblock In {\em KDD '03: Proceedings of the ninth ACM SIGKDD international
  conference on Knowledge discovery and data mining}, pages 147--156, New York,
  NY, USA, 2003. ACM.

\bibitem{KGR06}
Mehmet Koyut{\"u}rk, Ananth Grama, and Naren Ramakrishnan.
\newblock Nonorthogonal decomposition of binary matrices for bounded-error data
  compression and analysis.
\newblock {\em ACM Trans. Math. Software}, 32(1):33--69, 2006.

\bibitem{Golub2013}
Gene~H. Golub and Charles~F. Van~Loan.
\newblock {\em Matrix computations}.
\newblock Johns Hopkins Studies in the Mathematical Sciences. Johns Hopkins
  University Press, Baltimore, MD, fourth edition, 2013.

\bibitem{Kawamoto00}
T.~Kawamoto, K.~Hotta, T.~Mishima, J.~Fujiki, M.~Tanaka, and T.~Kurita.
\newblock Estimation of single tones from chord sounds using non-negative
  matrix factorization.
\newblock {\em Neural Network World}, 3:429--436, 2000.

\bibitem{Donoho03}
D.~Donoho and V.~Stodden.
\newblock When does nonnegative matrix factorization give a correct
  decomposition into parts?
\newblock In {\em Proc. 17th Ann. Conf. Neural Information Processing Systems},
  NIPS, Stanford University, Stanford, CA, 2003, 2003.

\bibitem{Chu2008}
Moody~T. Chu and Matthew~M. Lin.
\newblock Low-dimensional polytope approximation and its applications to
  nonnegative matrix factorization.
\newblock {\em SIAM J. Sci. Comput.}, 30(3):1131--1155, 2008.

\bibitem{Kim08}
Hyunsoo Kim and Haesun Park.
\newblock Nonnegative matrix factorization based on alternating nonnegativity
  constrained least squares and active set method.
\newblock {\em SIAM J. Matrix Anal. Appl.}, 30(2):713--730, 2008.

\bibitem{Kim2011}
Jingu Kim and Haesun Park.
\newblock Fast nonnegative matrix factorization: an active-set-like method and
  comparisons.
\newblock {\em SIAM J. Sci. Comput.}, 33(6):3261--3281, 2011.

\bibitem{Kim2014}
Jingu Kim, Yunlong He, and Haesun Park.
\newblock Algorithms for nonnegative matrix and tensor factorizations: a
  unified view based on block coordinate descent framework.
\newblock {\em J. Global Optim.}, 58(2):285--319, 2014.

\bibitem{Aloise2009}
Daniel Aloise, Amit Deshpande, Pierre Hansen, and Preyas Popat.
\newblock Np-hardness of euclidean sum-of-squares clustering.
\newblock {\em Machine Learning}, 75(2):245--248, 2009.

\bibitem{Dasgupta2009}
S.~Dasgupta and Y.~Freund.
\newblock Random projection trees for vector quantization.
\newblock {\em Information Theory, IEEE Transactions on}, 55(7):3229--3242,
  July 2009.

\bibitem{van1981}
P.~van Emde-Boas.
\newblock {\em Another NP-complete partition problem and the complexity of
  computing short vectors in a lattice}.
\newblock Report. Department of Mathematics. University of Amsterdam.
  Department, Univ., 1981.

\bibitem{Lenstra1982}
A.K. Lenstra, Jr. Lenstra, H.W., and L.~Lov\'{a}sz.
\newblock Factoring polynomials with rational coefficients.
\newblock {\em Mathematische Annalen}, 261(4):515--534, 1982.

\bibitem{Chang2009}
Xiao-Wen Chang and Gene~H. Golub.
\newblock Solving ellipsoid-constrained integer least squares problems.
\newblock {\em SIAM J. Matrix Anal. Appl.}, 31(3):1071--1089, 2009.

\bibitem{Foschini1999}
G.J. Foschini, G.D. Golden, R.A. Valenzuela, and P.W. Wolniansky.
\newblock Simplified processing for high spectral efficiency wireless
  communication employing multi-element arrays.
\newblock {\em Selected Areas in Communications, IEEE Journal on},
  17(11):1841--1852, Nov 1999.

\bibitem{Ling2007}
Cong Ling and N.~Howgrave-Graham.
\newblock Effective {LLL} reduction for lattice decoding.
\newblock In {\em Information Theory, 2007. ISIT 2007. IEEE International
  Symposium on}, pages 196--200, June 2007.

\bibitem{Xie2011}
Xiaohu Xie, Xiao-Wen Chang, and Mazen~Al Borno.
\newblock Partial {LLL} reduction.
\newblock In {\em Proceedings of IEEE GLOBECOM}, 2011.

\bibitem{Damen2003}
M.O. Damen, H.~El~Gamal, and G.~Caire.
\newblock On maximum-likelihood detection and the search for the closest
  lattice point.
\newblock {\em Information Theory, IEEE Transactions on}, 49(10):2389--2402,
  Oct 2003.

\bibitem{Fincke1985}
U.~Fincke and M.~Pohst.
\newblock Improved methods for calculating vectors of short length in a
  lattice, including a complexity analysis.
\newblock {\em Math. Comp.}, 44(170):463--471, 1985.

\bibitem{Wubben2011}
D.~W\"{u}bben, R.~Bohnke, J.~Rinas, V.~Kuhn, and K.-D. Kammeyer.
\newblock Efficient algorithm for decoding layered space-time codes.
\newblock {\em Electronics Letters}, 37(22):1348--1350, Oct 2001.

\bibitem{Chang2008}
Xiao-Wen Chang and Qing Han.
\newblock Solving box-constrained integer least squares problems.
\newblock {\em Wireless Communications, IEEE Transactions on}, 7(1):277--287,
  Jan 2008.

\bibitem{Breen2012}
Stephen Breen and Xiao-Wen Chang.
\newblock Column reordering for box-constrained integer least squares problems.
\newblock {\em http://arxiv.org/abs/1204.1407}, 2012.

\bibitem{Su2005}
K.~Su and I.J. Wassell.
\newblock A new ordering for efficient sphere decoding.
\newblock In {\em Communications, 2005. ICC 2005. 2005 IEEE International
  Conference on}, volume~3, pages 1906--1910 Vol. 3, May 2005.

\bibitem{Kannan2012}
R.~Kannan, M.~Ishteva, and Haesun Park.
\newblock Bounded matrix low rank approximation.
\newblock In {\em Data Mining (ICDM), 2012 IEEE 12th International Conference
  on}, pages 319--328, Dec 2012.

\end{thebibliography}

\end{document}